\documentclass[a4paper,12pt,oneside]{amsart}

\usepackage[applemac]{inputenc}
\usepackage[british]{babel}
\usepackage{fancyhdr}
\usepackage{geometry}
\usepackage{setspace}
\usepackage{bbm}
\usepackage{amssymb}
\usepackage[all]{xy}

\usepackage{color}
\definecolor{citation}{rgb}{0.2,0.3,1}
\usepackage[final,colorlinks,citecolor=citation,pagebackref,linktocpage]{hyperref}

\newlength{\aufzleft}
\newenvironment{aufz}{\begin{list}{}{\setlength{\listparindent}{0pt}\setlength{\itemsep}{\topsep}\setlength{\labelwidth}{3.2ex}\setlength{\aufzleft}{\labelsep}\addtolength{\aufzleft}{\labelwidth}\setlength{\leftmargin}{\aufzleft}}}{\end{list}}
\newenvironment{equi}{\begin{list}{}{\setlength{\listparindent}{0pt}\setlength{\itemsep}{\topsep}\setlength{\labelwidth}{4.1ex}\setlength{\aufzleft}{\labelsep}\addtolength{\aufzleft}{\labelwidth}\setlength{\leftmargin}{\aufzleft}}}{\end{list}}
\newenvironment{qulist}{\begin{list}{}{\it\setlength{\itemsep}{\topsep}\setlength{\labelwidth}{3em}\setlength{\aufzleft}{\labelsep}\addtolength{\aufzleft}{\labelwidth}\setlength{\leftmargin}{\aufzleft}}}{\end{list}}

\newtheoremstyle{par}{1ex}{2ex}{\rm}{}{\bfseries}{}{0.8em}{\thmnumber{(#2)}}
\newtheoremstyle{thm}{1ex}{2ex}{\itshape}{}{\bfseries}{}{0.9em}{\thmnumber{(#2)}\thmname{ #1}\thmnote{ (#3)}}
\newtheoremstyle{ex}{1ex}{2ex}{\rm}{}{\bfseries}{}{0.8em}{\thmnumber{(#2)}\thmname{ #1}}

\theoremstyle{par}
\newtheorem{no}{}[section]

\theoremstyle{thm}
\newtheorem{prop}[no]{Proposition}
\newtheorem{cor}[no]{Corollary}

\theoremstyle{ex}
\newtheorem{exas}[no]{Examples}

\newcommand{\N}{\mathbbm{N}}
\newcommand{\Z}{\mathbbm{Z}}
\newcommand{\dfgl}{\mathrel{\mathop:}=}
\newcommand{\ga}{\Gamma}
\newcommand{\oga}{\overline{\Gamma}}
\newcommand{\res}{\!\!\upharpoonright}
\newcommand{\Id}{{\rm Id}}
\newcommand{\catmod}{{\sf Mod}}
\newcommand{\ia}{\mathfrak{a}}
\newcommand{\ib}{\mathfrak{b}}
\newcommand{\ip}{\mathfrak{p}}
\newcommand{\iq}{\mathfrak{q}}
\newcommand{\im}{\mathfrak{m}}
\newcommand{\inn}{\mathfrak{n}}
\newcommand{\assf}{\ass^{\rm f}}
\renewcommand{\P}{{\mathbf P}}

\DeclareMathOperator{\ass}{Ass}
\DeclareMathOperator{\var}{Var}
\DeclareMathOperator{\spec}{Spec}
\DeclareMathOperator{\Max}{Max}
\DeclareMathOperator{\Min}{Min}
\DeclareMathOperator{\supp}{Supp}

\begin{document}

\newcommand{\leadingzero}[1]{\ifnum #1<10 0\the#1\else\the#1\fi} 
\renewcommand{\today}{\the\year\leadingzero{\month}\leadingzero{\day}}

\title{Assassins and torsion functors II}
\author{Fred Rohrer}
\address{Grosse Grof 9, 9470 Buchs, Switzerland}
\email{fredrohrer@math.ch}

\begin{abstract}
Fairness and centredness of ideals in commutative rings, i.e., the relations between assassins and weak assassins of a module, its small or large torsion submodule, and the corresponding quotients, are studied. General criteria as well as more specific results about idempotent or nil ideals are given, and several examples are presented.
\end{abstract}

\maketitle\thispagestyle{fancy}


\section*{Introduction}

Let $R$ be a ring\footnote{Throughout what follows, rings are understood to possess a unit element, not necessarily different from $0$, and to be commutative. In general, notation and terminology follow Bourbaki's \textit{\'El\'ements de math\'ematique.}}, and let $\ia\subseteq R$ be an ideal. For an $R$-module $M$ we consider its small $\ia$-torsion submodule \[\ga_\ia(M)=\{x\in M\mid\exists n\in\N\colon\ia^n\subseteq(0:_Rx)\}\] and its large $\ia$-torsion submodule \[\oga_\ia(M)=\{x\in M\mid\ia\subseteq\sqrt{(0:_Rx)}\}.\] This gives rise to subfunctors $\ga_\ia\hookrightarrow\oga_\ia\hookrightarrow\Id_{\catmod(R)}$. If $R$ is noetherian, the two torsion functors coincide, but they need not do so in general. For an $R$-module $M$ we consider its assassin \[\ass_R(M)=\{\ip\in\spec(R)\mid\exists x\in M\colon\ip=(0:_Rx)\}\] and its weak assassin \[\assf_R(M)=\{\ip\in\spec(R)\mid\exists x\in M\colon\ip\in\min(0:_Rx)\},\] where $\min(\ia)$ denotes the set of minimal primes of an ideal $\ia\subseteq R$. So, we get subsets $\ass_R(M)\subseteq\assf_R(M)\subseteq\spec(R)$. If $R$ is noetherian, the two subsets of $\spec(R)$ coincide, but they need not do so in general. In \cite{sol}, the functors $\ga_\ia$ and $\oga_\ia$ were investigated extensively. In \cite{asstor}, the relations between assassins and weak assassins of $\ga_\ia(M)$, $M$ and $M/\ga_\ia(M)$ were studied. The goal of this work is on one hand to extend the results from \cite{asstor}, and on the other hand to study their analogues for large torsion functors.

Torsion functors, and especially their right derived cohomological functors (i.e., local cohomology), are useful tools in commutative algebra and algebraic geometry. If the ring $R$ is noetherian, then they behave rather nicely (cf. \cite{bs} for a comprehensive treatment from an algebraic point of view). Several approaches to an extension of this theory to non-noetherian rings can be found in the literature. However, in general torsion functors quickly start to behave nastily; we refer the reader to \cite{lipman}, \cite{qr} and \cite{schenzel} for some examples.\medskip

In this article, we consider the following properties of an ideal $\ia$:
\begin{aufz}
\item[(1)] $\ia$ is fair, i.e., $\ass_R(M/\ga_\ia(M))=\ass_R(M)\setminus\var(\ia)$ for any $M$;
\item[($\overline{1}$)] $\ia$ is large fair, i.e., $\ass_R(M/\oga_\ia(M))=\ass_R(M)\setminus\var(\ia)$ for any $M$;
\item[(2)] $\ia$ is weakly fair, i.e., $\assf_R(M/\ga_\ia(M))=\assf_R(M)\setminus\var(\ia)$ for any $M$;
\item[($\overline{2}$)] $\ia$ is weakly large fair, i.e., $\assf_R(M/\oga_\ia(M))=\assf_R(M)\setminus\var(\ia)$ for any $M$;
\item[(3)] $\ia$ is weakly quasifair, i.e., $\assf_R(\ga_\ia(M))=\assf_R(M)\cap\var(\ia)$ for any $M$;
\item[($\overline{3}$)] $\ia$ is weakly large quasifair, i.e., $\assf_R(\oga_\ia(M))=\assf_R(M)\cap\var(\ia)$ for any $M$;
\item[(4)] $\ia$ is half-centred, i.e., $\assf_R(M)\cap\var(\ia)=\emptyset$ for any $M$ with $\ga_\ia(M)=0$;
\item[(5)] $\ia$ is centred, i.e., $\ga_\ia(M)=M$ for any $M$ with $\assf_R(M)\subseteq\var(\ia)$.
\end{aufz}
If $R$ is noetherian, all these statements hold, but none of them need hold in general. We are interested in conditions on $\ia$ under which some of these properties hold, as well as in relations among these properties. In \cite[4.5]{sol} it was shown for example that (4) holds if and only if $\ga_\ia=\oga_\ia$.

We will see that the ``large version'' of (4) holds always and that the ``large version'' of (5) is equivalent to ($\overline{3}$). We will also see that (2)$\Rightarrow$(3)$\Rightarrow$(5)$\Rightarrow$($\overline{3}$)$\Leftarrow$($\overline{2}$). A further result will be the equivalence of (3), (4) and (5) provided $\sqrt{\ia}$ is maximal and $\ga_\ia$ is a radical. Moreover, we will have a look at how the above properties behave when we manipulate $\ia$, e.g., by taking its radical, or by adding to it a further ideal with some of these properties.\medskip

After some rather general results and criteria in Section 2, we will have a closer look at two special classes of ideals, namely idempotent ideals (Section 3) and nil ideals (Section 4). Torsion functors with respect to such ideals behave not too bad. We will be able to show, for example, that idempotent ideals fulfil (1); as an application it will follow that any ideal in an absolutely flat ring has all of the above properties. The case of nil ideals is more complicated, but we will show that for the maximal ideal of a $0$-dimensional local ring there are at most five possibilities concerning fairness and centredness properties. (Unfortunately, for one of these five classes the author was not able to decide whether or not it is empty.) Finally, we will have a brief look at the behaviour of the above properties under localisation in Section 5.\medskip

\textbf{Notation.} We denote by $\spec(R)$ the spectrum of $R$, by $\Max(R)$ the set of maximal ideals of $R$, by $\Min(R)$ the set of minimal prime ideals of $R$, and by $\catmod(R)$ the category of $R$-modules. For a set $I$ we denote by $R[(X_i)_{i\in I}]$ the polynomial algebra over $R$ in the indeterminates $(X_i)_{i\in I}$. We denote by $\var(\ia)$ the variety of $\ia$, by $\min(\ia)$ the set of minimal elements of $\var(\ia)$, and by $\max(\ia)$ the set of maximal element of $\var(\ia)$. For an $R$-module $M$ we denote by $\supp_R(M)$ the support of $M$.


\section{Preliminaries}

We collect basic facts about torsion functors, assassins, and weak assassins. For details we refer the reader to \cite{sol}, \cite[Chapter IV]{ac} (especially Exercice IV.1.17) and \cite[00L9, 0546]{stacks}.

\begin{no}\label{pre03}
Setting \[\ga_\ia(M)\dfgl\{x\in M\mid\exists n\in\N\colon\ia^n\subseteq(0:_Rx)\}\] and \[\oga_\ia(M)\dfgl\{x\in M\mid\ia\subseteq\sqrt{(0:_Rx)}\}\] for every $R$-module $M$ yields subfunctors \[\ga_\ia\hookrightarrow\oga_\ia\hookrightarrow\Id_{\catmod(R)}\] (\cite[3.2]{sol}). The functors  $\ga_\ia$ and $\oga_\ia$ are called \textit{the small $\ia$-torsion functor} and \textit{the large $\ia$-torsion functor;} they need not be equal (\cite[Section 4]{sol}).
\end{no}

\begin{no}\label{pre05}
A) Both functors $\ga_\ia$ and $\oga_\ia$ are left exact, hence for an $R$-module $M$ and a sub-$R$-module $N\subseteq M$ we have $N\cap\ga_\ia(M)=\ga_\ia(N)$ and $N\cap\oga_\ia(M)=\oga_\ia(N)$ (\cite[9.1 A), 9.2]{sol}).\smallskip

B) We have $\ga_\ia=\Id_{\catmod(R)}$ if and only if $\ia$ is nilpotent, and $\oga_\ia=\Id_{\catmod(R)}$ if and only if $\ia$ is nil. Moreover, $\ga_\ia=0$ if and only if $\oga_\ia=0$ if and only if $\ia=R$ (\cite[3.6]{sol}).\smallskip

C) If $\ib\subseteq R$ is an ideal with $\ia\subseteq\ib$, then $\ga_\ib$ and $\oga_\ib$ are subfunctors of $\ga_\ia$ and $\oga_\ia$, resp. Moreover, if $\ib\subseteq R$ is an arbitrary ideal, then $\oga_\ia=\oga_\ib$ if and only if $\sqrt{\ia}=\sqrt{\ib}$. Furthermore, if $n\in\N^*$, then $\ga_\ia=\ga_{\ia^n}$ and $\oga_\ia=\oga_{\ia^n}$ (\cite[3.4, 3.5]{sol}).\smallskip

D) If $\ib\subseteq R$ is an ideal, then $\ga_{\ia+\ib}=\ga_\ia\circ\ga_\ib$ and $\oga_{\ia+\ib}=\oga_\ia\circ\oga_\ib$ (\cite[3.3 D)]{sol}).\smallskip

E) The functor $\oga_\ia$ is a radical, i.e., $\oga_\ia(M/\oga_\ia(M))=0$ for every $R$-module $M$. If $R$ is noetherian or $\ia$ is idempotent, then $\ga_\ia$ is a radical; it need not be so in general, but it may be so even if $\ga_\ia\neq\oga_\ia$ (\cite[Section 5]{sol}).\smallskip

F) If $\ip\in\spec(R)\setminus\var(\ia)$, then $\ga_\ia(R/\ip)=\oga_\ia(R/\ip)=0$. If $\ib\subseteq R$ is an ideal with $\ia\subseteq\ib$, then $\ga_\ia(R/\ib)=\oga_\ia(R/\ib)=R/\ib$. If $M$ is an $R$-module, then $\supp_R(\ga_\ia(M))\subseteq\supp_R(\oga_\ia(M))\subseteq\var(\ia)$.\smallskip

G) If $R\rightarrow S$ is a morphism of rings, then $\ga_\ia(\bullet\res_R)=\ga_{\ia S}(\bullet)\res_R$ and $\oga_\ia(\bullet\res_R)=\oga_{\ia S}(\bullet)\res_R$ as functors from $\catmod(S)$ to $\catmod(R)$ (\cite[3.7 A)]{sol}).
\end{no}

\begin{no}\label{pre06}
An $R$-module $M$ is said to be \textit{of bounded small $\ia$-torsion} if there exists $n\in\N$ with $\ia^n\ga_\ia(M)=0$, and \textit{of bounded large $\ia$-torsion} if there exists $n\in\N$ with $\ia^n\oga_\ia(M)=0$. If $M$ is of bounded large $\ia$-torsion, then $\ga_\ia(M)=\oga_\ia(M)$. If $\ia$ is idempotent, then every $R$-module is of bounded small $\ia$-torsion (\cite[7.4 A), 7.5 a)]{sol}).
\end{no}

\begin{no}\label{pre08}
Let $M$ be an $R$-module. A prime ideal $\ip\subseteq R$ is said to be \textit{associated to $M$} if there exists $x\in M$ with $\ip=(0:_Rx)$, and \textit{weakly associated to $M$} if there exists $x\in M$ with $\ip\in\min(0:_Rx)$. The sets $\ass_R(M)$ and $\assf_R(M)$ of associated and weakly associated primes of $M$ are called \textit{the assassin of $M$} and \textit{the weak assassin of $M$} (\cite[IV.1]{ac}).
\end{no}

\begin{no}\label{pre10}
A) Let $M$ be an $R$-module. We have $\ass_R(M)\subseteq\assf_R(M)$, with equality if $R$ or $M$ is noetherian. Furthermore, $M=0$ if and only if $\assf_R(M)=\emptyset$ (\cite[0589, 058A, 0588]{stacks}, \cite[1.2]{yassemi}).\smallskip

B) If $0\rightarrow L\rightarrow M\rightarrow N\rightarrow 0$ is an exact sequence of $R$-modules, then \[\ass_R(L)\subseteq\ass_R(M)\subseteq\ass_R(L)\cup\ass_R(N)\] and \[\assf_R(L)\subseteq\assf_R(M)\subseteq\assf_R(L)\cup\assf_R(N)\] (\cite[02M3, 0548]{stacks}).\smallskip

C) We have $\assf_R(R/\ia)\subseteq\var(\ia)$. If $\ip\in\spec(R)$, then $\ass_R(R/\ip)=\assf_R(R/\ip)$\linebreak$=\{\ip\}$.\smallskip

D) If $M$ is an $R/\ia$-module, then the isomorphism of ordered sets \[\var(\ia)\rightarrow\spec(R/\ia),\;\ip\mapsto\ip/\ia\] induces by restriction and coastriction bijections \[\assf_R(M\res_R)\rightarrow\assf_{R/\ia}(M)\quad\text{and}\quad\ass_R(M\res_R)\rightarrow\ass_{R/\ia}(M)\] (\cite[05BY, 05C8]{stacks}).
\end{no}

\begin{no}\label{pre11}
A) For a subset $S\subseteq R$ there are canonical monomorphisms of functors \[\rho^S_\ia\colon S^{-1}\ga_\ia(\bullet)\rightarrowtail\ga_{S^{-1}\ia}(S^{-1}\bullet)\quad\text{and}\quad\overline{\rho}^S_\ia\colon S^{-1}\oga_\ia(\bullet)\rightarrowtail\oga_{S^{-1}\ia}(S^{-1}\bullet)\] that need not be isomorphisms (\cite[8.2, 8.3 B)]{sol}). If $S=R\setminus\ip$ for some $\ip\in\spec(R)$, then they are also denoted by $\rho^\ip_\ia$ and $\overline{\rho}^\ip_\ia$.\smallskip

B) For a subset $S\subseteq R$ there is a bijection \[\{\ip\in\assf_R(M)\mid\ip\cap S=\emptyset\}\rightarrow\assf_{S^{-1}R}(S^{-1}M),\;\ip\mapsto S^{-1}\ip.\] By restriction and coastriction it induces an injection \[\{\ip\in\ass_R(M)\mid\ip\cap S=\emptyset\}\rightarrow\ass_{S^{-1}R}(S^{-1}M)\] that need not be surjective (\cite[05C9, 05BZ, ]{stacks}, \cite[IV.1 Exercice 1 c)]{ac}). Note that if $N$ is an $S^{-1}R$-module, then $\{\ip\in\ass_R(N\res_R)\mid\ip\cap S=\emptyset\}=\ass_R(N\res_R)$.
\end{no}

\begin{prop}\label{pre12}
If $M$ is an $R$-module and $N\subseteq M$ is a sub-$R$-module, then\/\footnote{The statement about assassins is part of \cite[IV.1 Exercise 3]{ac}.} \[\ass_R(M/N)\setminus\var(0:_RN)\subseteq\ass_R(M)\text{ and }\assf_R(M/N)\setminus\var(0:_RN)\subseteq\assf_R(M).\]
\end{prop}

\begin{proof}
Let $\ip\in\ass_R(M/N)\setminus\var(0:_RN)$. There exist $x\in M\setminus N$ with $\ip=(N:_Rx)$ and $r\in R\setminus\ip$ with $rN=0$. If $s\in(0:_Rrx)$, then $srx=0\in N$, hence $sr\in\ip$, and thus $s\in\ip$. If $t\in\ip$, then $trx=rtx\in rN=0$, and thus $t\in(0:_Rrx)$. This shows that $\ip=(0:_Rrx)$. It follows that $\ip\in\ass_R(M)$.

Let $\ip\in\assf_R(M/N)\setminus\var(0:_RN)$. There exist $x\in M\setminus N$ with $\ip\in\min(N:_Rx)$ and $r\in R\setminus\ip$ with $rN=0$. If $s\in(0:_Rrx)$, then $srx=0\in N$, hence $sr\in\ip$, and thus $s\in\ip$. This shows that $(0:_Rrx)\subseteq\ip$. Let $\iq\in\spec(R)$ with $(0:_Rrx)\subseteq\iq\subseteq\ip$. If $s\in(N:_Rx)$, then $sx\in N$, hence $srx=rsx=0$, and thus $s\in(0:_Rrx)\subseteq\iq$. It follows that $(N:_Rx)\subseteq\iq$, and minimality of $\ip$ implies $\iq=\ip$. Therefore, $\ip\in\min(0:_Rrx)$, and thus $\ip\in\assf_R(M)$.
\end{proof}


\section{Fairness and centredness}

In this section, we recall the fairness and centredness properties introduced in \cite{asstor}. Moreover, we introduce ``large versions'' of these fairness properties, and we show that ``large versions'' of these centredness properties yield nothing new. In the further results, there are three main themes. First, we are interested in how these properties behave under change of the supporting ideal. Second, we look for implications between these properties. And third, we collect criteria for some of these properties.

\newpage
\begin{no}\label{fac10}
A) Let $M$ be an $R$-module. By \cite[3.1]{asstor}, we have the following relations.
\begin{aufz}
\item[a)] $\ass_R(\ga_\ia(M))=\ass_R(M)\cap\var(\ia)$;
\item[b)] $\assf_R(\ga_\ia(M))\subseteq\assf_R(M)\cap\var(\ia)$;
\item[c)] $\ass_R(M/\ga_\ia(M))\supseteq\ass_R(M)\setminus\var(\ia)$;
\item[d)] $\assf_R(M/\ga_\ia(M))\supseteq\assf_R(M)\setminus\var(\ia)$.
\end{aufz}
The $R$-module $M$ is called \textit{weakly $\ia$-quasifair} if \[\assf_R(\ga_\ia(M))=\assf_R(M)\cap\var(\ia),\] \textit{$\ia$-fair} if \[\ass_R(M/\ga_\ia(M))=\ass_R(M)\setminus\var(\ia),\] and \textit{weakly $\ia$-fair} if \[\assf_R(M/\ga_\ia(M))=\assf_R(M)\setminus\var(\ia).\] (In view of \ref{fac30}, these notions could be called ``weakly small $\ia$-quasifair'', ``small $\ia$-fair'' and ``weakly small $\ia$-fair'', but we stick to the less clumsy terminology introduced in \cite{asstor}.)\smallskip

B) The ideal $\ia$ is called \textit{weakly quasifair,} \textit{weakly fair,} or \textit{fair,} resp. if every $R$-module is weakly $\ia$-quasifair, weakly $\ia$-fair, or $\ia$-fair, resp.\smallskip

C) By \cite[3.5]{asstor}, $\ia$ is weakly quasifair, weakly fair, or fair resp. if and only if every monogeneous $R$-module is weakly $\ia$-quasifair, weakly $\ia$-fair, or $\ia$-fair, resp.
\end{no}

\begin{prop}\label{fac20}
Let $M$ be an $R$-module. Then:
\begin{aufz}
\item[a)] $\ass_R(\oga_\ia(M))=\ass_R(M)\cap\var(\ia)=\ass_R(\ga_\ia(M))$;
\item[b)] $\assf_R(\oga_\ia(M))\subseteq\assf_R(M)\cap\var(\ia)$;
\item[c)] $\ass_R(M/\oga_\ia(M))\supseteq\ass_R(M)\setminus\var(\ia)$;
\item[d)] $\assf_R(M/\oga_\ia(M))\supseteq\assf_R(M)\setminus\var(\ia)$.
\end{aufz}
\end{prop}

\begin{proof}
We have $\ass_R(\oga_\ia(M))\subseteq\ass_R(M)$ and $\assf_R(\oga_\ia(M))\subseteq\assf_R(M)$ (\ref{pre10} B)). Let $\ip\in\assf_R(\oga_\ia(M))$. There exists $x\in\oga_\ia(M)$ with $(0:_Rx)\subseteq\ip$. For $r\in\ia$ there exists $n\in\N$ with $r^nx=0$, hence $r^n\subseteq(0:_Rx)\subseteq\ip$ and therefore $r\in\ip$. It follows $\ia\subseteq\ip$, hence $\ass_R(\oga_\ia(M))\subseteq\assf_R(\oga_\ia(M))\subseteq\var(\ia)$ (\ref{pre10} A)). So, we have proven b) and the inclusion ``$\subseteq$'' at the first place in a). As \[\ass_R(M)\cap\var(\ia)=\ass_R(\ga_\ia(M))\subseteq\ass_R(\oga_\ia(M))\] (\ref{fac10} A) a), \ref{pre03}, \ref{pre10} B)) we also get the inclusion ``$\supseteq$'' at the first place and the second equality in a). Finally, \[\ass_R(M)\subseteq\ass_R(\oga_\ia(M))\cup\ass_R(M/\oga_\ia(M))\] and \[\assf_R(M)\subseteq\assf_R(\oga_\ia(M))\cup\assf_R(M/\oga_\ia(M))\] (\ref{pre10} B)), thus c) and d) follow from a) and b).
\end{proof}

\begin{no}\label{fac30}
A) An $R$-module $M$ is called \textit{weakly large $\ia$-quasifair} if \[\assf_R(\oga_\ia(M))=\assf_R(M)\cap\var(\ia),\] \textit{large $\ia$-fair} if \[\ass_R(M/\oga_\ia(M))=\ass_R(M)\setminus\var(\ia),\] and \textit{weakly large $\ia$-fair} if \[\assf_R(M/\oga_\ia(M))=\assf_R(M)\setminus\var(\ia).\]

B) The ideal $\ia$ is called \textit{weakly large quasifair,} \textit{weakly large fair,} or \textit{large fair}, resp. if every $R$-module is weakly large $\ia$-quasifair, weakly large $\ia$-fair, or large $\ia$-fair, resp.
\end{no}

\begin{prop}\label{fac40}
The ideal $\ia$ is weakly large quasifair, weakly large fair, or large fair, resp. if and only if every monogeneous $R$-module is weakly large $\ia$-quasifair, weakly large $\ia$-fair, or large $\ia$-fair, resp.
\end{prop}

\begin{proof} 
Let $M$ be an $R$-module. Suppose that every monogeneous $R$-module is weakly large $\ia$-quasifair. Let $\ip\in\assf_R(M)\cap\var(\ia)$. There exists $x\in M$ with $\ip\in\min(0:_Rx)$. It follows that \[\ip\in\assf_R(\langle x\rangle_R)\cap\var(\ia)=\assf_R(\oga_\ia(\langle x\rangle_R))\subseteq\assf_R(\oga_\ia(M))\] (\ref{pre10} B), \ref{pre05} A)), and thus $M$ is weakly large $\ia$-quasifair (\ref{fac20} b)).

Next, suppose that every monogeneous $R$-module is large $\ia$-fair or weakly large $\ia$-fair, resp. Let $\ip\in\ass_R(M/\oga_\ia(M))$ or $\ip\in\assf_R(M/\oga_\ia(M))$, resp. There exists $x\in M$ with $\ip=(0:_R\overline{x})$ or $\ip\in\min(0:_R\overline{x})$, resp., where $\overline{x}$ denotes the canonical image of $x$ in $M/\oga_\ia(M)$. Then, $\langle x\rangle_R/\oga_\ia(\langle x\rangle_R)\cong\langle\overline{x}\rangle_R$ (\ref{pre05} A)), and thus \[\ip\in\ass_R(\langle\overline{x}\rangle_R)=\ass_R(\langle x\rangle_R/\oga_\ia(\langle x\rangle_R))=\ass_R(\langle x\rangle_R)\setminus\var(\ia)\subseteq\ass_R(M)\setminus\var(\ia)\] or \[\ip\in\assf_R(\langle\overline{x}\rangle_R)=\assf_R(\langle x\rangle_R/\oga_\ia(\langle x\rangle_R))=\assf_R(\langle x\rangle_R)\setminus\var(\ia)\subseteq\assf_R(M)\setminus\var(\ia),\] resp. (\ref{pre10} B)). Thus, $M$ is large $\ia$-fair or weakly large $\ia$-fair, resp. (\ref{fac20} c), d)). 
\end{proof}

\begin{prop}\label{fac60}
Let $M$ be an $R$-module. Then:
\begin{aufz}
\item[a)] $\assf_R(M)\cap\var(\ia)=\emptyset\Rightarrow\oga_\ia(M)=0\Rightarrow\ga_\ia(M)=0\Rightarrow\ass_R(M)\cap\var(\ia)=\emptyset$;
\item[b)] $\ga_\ia(M)=M\Rightarrow\assf_R(M)\subseteq\var(\ia)\Leftrightarrow\oga_\ia(M)=M$.
\end{aufz}
\end{prop}

\begin{proof}
a) The first implication follows from \ref{fac20} b) and \ref{pre10} A), the second from \ref{pre03}, and the third from \ref{fac10} a). b) holds by \cite[4.2]{sol}.
\end{proof}

\begin{no}\label{fac63}
A) The ideal $\ia$ is called \textit{centred} if \[\ga_\ia(M)=0\Leftrightarrow\assf_R(M)\cap\var(\ia)=\emptyset\] for every $R$-module $M$, \textit{half-centred} if \[\ga_\ia(M)=M\Leftrightarrow\assf_R(M)\subseteq\var(\ia)\] for every $R$-module $M$, and \textit{well-centred} if it is centred and half-centred. (These notions could be called ``small centred'', ``small half-centred'' and ``small well-centred'', but we will recognise this as superfluous in \ref{fac110}.)

B) By \cite[4.5]{sol}, $\ia$ is half-centred if and only if $\ga_\ia=\oga_\ia$.\smallskip

C) By B), \ref{pre05} E) and \cite[4.4]{asstor}, $\ia$ is well-centred if and only if $\ia$ is centred and $\ga_\ia$ is a radical.
\end{no}

\begin{prop}\label{fac70}
a) The ideal $\ia$ is centred if and only if $\assf_R(M)\cap\var(\ia)=\emptyset$ for every monogeneous $R$-module $M$ with $\ga_\ia(M)=0$.

b) The ideal $\ia$ is half-centred if and only if $\ga_\ia(M)=M$ for every monogeneous $R$-module $M$ with $\assf_R(M)\subseteq\var(\ia)$.
\end{prop}

\begin{proof}
a) Suppose that $\assf_R(M)\cap\var(\ia)=\emptyset$ for every monogeneous $R$-module $M$ with $\ga_\ia(M)=0$. Let $M$ be an $R$-module with $\ga_\ia(M)=0$. Let $\ip\in\assf_R(M)$. There exists $x\in M$ with $\ip\in\min(0:_Rx)$. Then, $\ga_\ia(\langle x\rangle_R)=0$ (\ref{pre05} A)) and $\ip\in\assf_R(\langle x\rangle_R)$, hence $\ip\notin\var(\ia)$. It follows that $\ia$ is centred. The converse is clear.

b) Suppose that $\ga_\ia(M)=M$ for every monogeneous $R$-module $M$ with\linebreak $\assf_R(M)\subseteq\var(\ia)$. Let $M$ be an $R$-module with $\assf_R(M)\subseteq\var(\ia)$. Let $x\in M$. Then, $\assf_R(\langle x\rangle_R)\subseteq\assf_R(M)\subseteq\var(\ia)$ (\ref{pre10} B)), hence $\ga_\ia(\langle x\rangle_R)=\langle x\rangle_R$, and therefore $x\in\ga_\ia(M)$ (\ref{pre05} A)). This shows that $\ga_\ia(M)=M$. It follows that $\ia$ is half-centred. The converse is clear.
\end{proof}

\begin{prop}\label{fac100}
The following statements are equivalent:
\begin{equi}
\item[(i)] $\ia$ is weakly large quasifair;
\item[(ii)] $\oga_\ia(M)=0\Leftrightarrow\assf_R(M)\cap\var(\ia)=\emptyset$ for every $R$-module $M$.
\end{equi}
\end{prop}

\begin{proof}
If (i) holds and $M$ is an $R$-module with $\oga_\ia(M)=0$, then\linebreak $\assf_R(M)\cap\var(\ia)=\assf_R(\oga_\ia(M))=\emptyset$, hence \ref{fac60} a) yields (ii). If (ii) holds and $M$ is an $R$-module, then $\assf_R(M/\oga_\ia(M))\cap\var(\ia)=\emptyset$ (\ref{pre05} E)), hence \[\assf_R(M)\cap\var(\ia)\subseteq\assf_R(\oga_\ia(M))\cup(\assf_R(M/\oga_\ia(M))\cap\var(\ia))=\assf_R(\oga_\ia(M))\] (\ref{pre10} B)), and thus \ref{fac20} b) yields (i).
\end{proof}

\begin{no}\label{fac110}
As every ideal is ``large half-centred'' by \ref{fac60} b) and ``large centredness'' is equivalent to weak large quasifairness by \ref{fac100}, there is no need to introduce large centredness notions. Therefore, we will stick to the terminology for small centredness introduced in \cite{asstor}.
\end{no}

\begin{no}\label{fac115}
A) Let $\ib\subseteq R$ be an ideal and let $M$ be an $R/\ib$-module. It follows from \ref{pre10} D) and \ref{pre05} G) that $M$ is (large) $(\ia+\ib)/\ib$-fair, weakly (large) $(\ia+\ib)/\ib$-fair, or weakly (large) $(\ia+\ib)/\ib$-quasifair, resp. if and only if $M\res_R$ is (large) $\ia$-fair, weakly (large) $\ia$-fair, or weakly (large) $\ia$-quasifair, resp.\smallskip

B) Let $\ib\subseteq R$ be an ideal. It follows from A) that if $\ia$ is (large) fair, weakly (large) fair, or weakly (large) quasifair, resp., then so is $(\ia+\ib)/\ib$.\smallskip

C) Let $\ib\subseteq R$ be an ideal. It follows from \ref{pre10} D) and \ref{pre05} G) that if $\ia$ is centred, half-centred, or well-centred, resp., then so is $(\ia+\ib)/\ib$.\smallskip

D) Let $\ib\subseteq R$ be an ideal. It follows from C) and \ref{fac63} B) that if $\ga_\ia=\oga_\ia$, then $\ga_{(\ia+\ib)/\ib}=\oga_{(\ia+\ib)/\ib}$.
\end{no}

\begin{prop}\label{fac120}
For an $R$-module $M$ we have the following implications: \[\xymatrix@R15pt{M\text{ is weakly }\ia\text{-fair}\ar@{=>}[d]&M\text{ is weakly large }\ia\text{-fair}\ar@{=>}[d]\\M\text{ is weakly }\ia\text{-quasifair}\ar@{=>}[r]&M\text{ is weakly large }\ia\text{-quasifair}.}\]
\end{prop}

\begin{proof}
The left vertical implication holds by \cite[3.3]{asstor}. For a weakly large $\ia$-fair $R$-module $M$ we have \[\assf_R(M)\cap\var(\ia)\subseteq\assf_R(\oga_\ia(M))\cup(\assf_R(M/\oga_\ia(M))\cap\var(\ia))=\]\[\assf_R(\oga_\ia(M))\cup((\assf_R(M)\setminus\var(\ia))\cap\var(\ia))=\assf_R(\oga_\ia(M))\] (\ref{pre10} B)), so the right vertical implication follows from \ref{fac20} b). For a weakly $\ia$-quasifair $R$-module $M$ we have \[\assf_R(M)\cap\var(\ia)=\assf_R(\ga_\ia(M))\subseteq\assf_R(\oga_\ia(M))\subseteq\assf_R(M)\cap\var(\ia)\] (\ref{pre10} B), \ref{pre03}), so the horizontal implication follows from \ref{fac20} b).
\end{proof}

\begin{prop}\label{fac130}
a) We have the following implications: \[\xymatrix@R15pt{\ia\text{ is weakly fair}\ar@{=>}[d]&&\ia\text{ is weakly large fair}\ar@{=>}[d]\\\ia\text{ is weakly quasifair}\ar@{=>}[r]&\ia\text{ is centred}\ar@{=>}[r]&\ia\text{ is weakly large quasifair}.}\]

b) If $\ia$ is half-centred, we have the following implications: \[\xymatrix@R15pt{\ia\text{ is weakly fair}\qquad\ar@{<=>}[rr]\ar@{=>}[d]&&\qquad\ia\text{ is weakly large fair}\ar@{=>}[d]\\\ia\text{ is weakly quasifair}\ar@{<=>}[r]&\ia\text{ is centred}\ar@{<=>}[r]&\ia\text{ is weakly large quasifair}.}\]
\end{prop}

\begin{proof}
a) The vertical implications are clear by \ref{fac120}. The first horizontal implication was shown in \cite[4.4]{asstor} under the hypothesis that $\ga_\ia$ is a radical, but no use was made of this hypothesis. The second horizontal implication follows from \ref{fac60} a) and \ref{fac100}. b) follows from a) and \ref{fac63} B).
\end{proof}

\begin{no}\label{fac131}
A) Let $\ib\subseteq R$ be an ideal with $\ia\subseteq\ib\subseteq\sqrt{\ia}$ (e.g. $\ia=\ib^n$ for some $n\in\N^*$, or $\ib=\sqrt{\ia}$). Then, $\var(\ia)=\var(\ib)$ and $\oga_\ia=\oga_\ib$ (\ref{pre05} C)). Therefore, an $R$-module $M$ is large $\ia$-fair, weakly large $\ia$-fair, or weakly large $\ia$-quasifair, resp. if and only if it is large $\ib$-fair, weakly large $\ib$-fair, or weakly large $\ib$-quasifair, resp. In particular, $\ia$ is large fair, weakly large fair, or weakly large quasifair if and only if $\ib$ is so.\smallskip

B) Let $n\in\N^*$. Then, $\ga_\ia=\ga_{\ia^n}$ (\ref{pre05} C)). Therefore, an $R$-module $M$ is $\ia$-fair, weakly $\ia$-fair, or weakly $\ia$-quasifair, resp. if and only if it is $\ia^n$-fair, weakly $\ia^n$-fair, or weakly $\ia^n$-quasifair, resp. In particular, $\ia$ is fair, weakly fair, or weakly quasifair if and only if $\ia^n$ is so. Moreover, $\ia$ is half-centred, centred, or well-centred if and only if $\ia^n$ is so.
\end{no}

\begin{prop}\label{fac132}
Let $\ib\subseteq R$ be an ideal with $\ia\subseteq\ib\subseteq\sqrt{\ia}$.

a) A weakly $\ib$-quasifair $R$-module is weakly $\ia$-quasifair.

b) If $\ib$ is weakly quasifair, then so is $\ia$.

c) If $\ib$ is half-centred, centred, or well-centred, then so is $\ia$.
\end{prop}

\begin{proof}
a) For a weakly $\ib$-quasifair $R$-module $M$ we have \[\assf_R(M)\cap\var(\ia)=\assf_R(M)\cap\var(\ib)=\assf_R(\ga_{\ib}(M))\subseteq\assf_R(\ga_\ia(M)),\] and hence we get the claim. b) follows from a). c) holds since $\var(\ia)=\var(\ib)$ and $\ga_\ib$ is a subfunctor of $\ga_\ia$.
\end{proof}

\begin{no}\label{fac133}
The converses of \ref{fac132} a) and b) need not hold. More precisely, if an $R$-module $M$ is $\ia$-fair, weakly $\ia$-fair, or weakly $\ia$-quasifair, resp., then it need not be $\sqrt{\ia}$-fair, weakly $\sqrt{\ia}$-fair, or weakly $\sqrt{\ia}$-quasifair, resp. Indeed, since every $R$-module is $0$-fair and weakly $0$-fair, it suffices to exhibit $0$-dimensional local rings whose maximal ideals, necessarily equal to $\sqrt{0}$, are not fair or not weakly quasifair (\ref{fac130} a)). Such examples were constructed in the proofs of \cite[3.8, 3.9]{asstor}. (In \ref{nil60} B) we will see that the converses of \ref{fac132} c) hold neither.)
\end{no}

\begin{prop}\label{fac134}
a) If $\ia$ has a power of finite type, then it is well-centred and weakly quasifair.

b) Ideals in noetherian rings are well-centred, fair, and weakly fair.

c) Noetherian $R$-modules are $\ia$-fair, weakly $\ia$-fair, large $\ia$-fair, and weakly large $\ia$-fair.
\end{prop}

\begin{proof}
a) By \ref{fac131} B), we may suppose that $\ia$ is of finite type. Then, $\ia$ is half-centred (\cite[4.4 c)]{sol}), hence $\ga_\ia=\oga_\ia$ is a radical (\ref{fac63} B), \ref{pre05} E)), and $\ia$ is weakly quasifair (\cite[5.4]{asstor}), thus well-centred (\ref{fac130} a), \ref{fac63} B)). b) follows from \cite[3.6]{asstor} and a). c) follows from \cite[3.6]{asstor} and \cite[4.6 d)]{sol}.
\end{proof}

\begin{exas}\label{fac136}
A) The ideal $R$ is well-centred, fair, and weakly fair (\ref{pre05} B)).\smallskip

B) An $R$-module $M$ with $\oga_\ia(M)=M$ is large $\ia$-fair and weakly large $\ia$-fair; an $R$-module $M$ with $\ga_\ia(M)=M$ is $\ia$-fair, weakly $\ia$-fair, large $\ia$-fair, and weakly large $\ia$-fair (\ref{pre03}).\smallskip

C) Nil ideals are large fair and weakly large fair; nilpotent ideals are well-centred, fair, and weakly fair (B), \ref{pre05} B), \ref{fac134} a)).\smallskip

D) If $\ib\subseteq R$ is an ideal with $\ia\subseteq\ib$, then the $R$-module $R/\ib$ is $\ia$-fair, weakly $\ia$-fair, large $\ia$-fair, and weakly large $\ia$-fair (B), \ref{pre05} F)).\smallskip

E) If $\ip\in\spec(R)$, the $R$-module $R/\ip$ is $\ia$-fair, weakly $\ia$-fair, large $\ia$-fair, and weakly large $\ia$-fair. Indeed, for $\ip\in\var(\ia)$ this holds by D), and for $\ip\notin\var(\ia)$ it follows from \ref{pre10} C) and \ref{pre05} F). In particular, if $R$ is integral, then the $R$-module $R$ is $\ia$-fair, weakly $\ia$-fair, large $\ia$-fair, and weakly large $\ia$-fair.
\end{exas}

\begin{no}\label{fac137}
If we wish to check whether $\ia$ has some fairness or centredness property, then by \ref{fac10} C), \ref{fac40} and \ref{fac70}, it suffices to consider $R$-modules of the form $R/\ib$ for ideals $\ib\subseteq R$. By \ref{fac136} D) and \ref{pre05} F), we may additionally suppose that $\ia\not\subseteq\ib$.
\end{no}

\begin{no}\label{fac161}
In \cite[5.11 $(*)$]{asstor} we asked whether every ideal $\ia$ such that $\ga_\ia$ is a radical is weakly quasifair, or even weakly fair. The answer to this is negative. Indeed, by \cite[5.5 B)]{sol} and \ref{fac63} B), there exist a ring $R$ and an ideal $\ia\subseteq R$ that is not half-centred such that $\ga_\ia$ is a radical. Then, $\ia$ is not well-centred, hence not weakly quasifair, and thus not weakly fair (\ref{fac130} a), \ref{fac63} B)).
\end{no}

\begin{prop}\label{fac162}
a) Suppose that $\ga_\ia(R)=\ia$. If $R\neq 0$, the $R$-module $R$ is not weakly $\ia$-fair. If $\ia$ is prime, the $R$-module $R$ is neither $\ia$-fair nor weakly $\ia$-fair.

b) Suppose that $\oga_\ia(R)=\ia$. If $R\neq 0$, the $R$-module $R$ is not weakly large $\ia$-fair. If $\ia$ is prime, the $R$-module $R$ is neither large $\ia$-fair nor weakly large $\ia$-fair.
\end{prop}

\begin{proof}
Let $F$ denote $\ga_\ia$ or $\oga_\ia$. As $F(R)=\ia$, we have \[\ass_R(R/F(R))\subseteq\assf_R(R/F(R))=\assf_R(R/\ia)\subseteq\var(\ia)\] (\ref{pre10} A), C)). So, if the $R$-module $R$ is weakly (large) $\ia$-fair, then $\assf_R(R/\ia)=\emptyset$, hence $R/\ia=0$, thus $\ia=R$ and therefore $R=0$. Moreover, if the $R$-module $R$ is (large) $\ia$-fair, then $\ass_R(R/\ia)=\emptyset$, and thus $\ia$ is not prime (\ref{pre10} C)).
\end{proof}

\begin{prop}\label{fac163}
Let $M$ be an $R$-module.

a) $\ass_R(M/\oga_{\ia}(M))\cap\var(\ia)=\emptyset$.

b) If $\ga_\ia$ is a radical, then $\ass_R(M/\ga_{\ia}(M))\cap\var(\ia)=\emptyset$.

c) If $\ia$ is well-centred, then $\assf_R(M/\ga_{\ia}(M))\cap\var(\ia)=\emptyset$.\footnote{By \ref{fac134} a) this generalises \cite[5.3]{asstor}.}
\end{prop}

\begin{proof}
a) Applying \ref{pre05} E) and \ref{fac20} a) to the $R$-module $M/\oga_\ia(M)$ yields \[\emptyset=\ass_R(0)=\ass_R(\oga_\ia(M/\oga_\ia(M)))=\ass_R(M/\oga_\ia(M))\cap\var(\ia)\] (\ref{pre05} E)) and thus the claim. b) holds by \cite[4.1]{asstor}. c) As $\ga_\ia$ is a radical (\ref{fac63} C)), we have $\ga_\ia(M/\ga_\ia(M))=0$, so centredness implies $\assf_R(M/\ga_{\ia}(M))\cap\var(\ia)=\emptyset$.
\end{proof}

\begin{prop}\label{fac165}
Let $\ib\subseteq R$ be an ideal.

a) If $\ia$ and $\ib$ are half-centred, then so is $\ia+\ib$.

b) If $\ia$ is centred and $\ib$ is well-centred or weakly fair, then $\ia+\ib$ is centred.

c) If $\ia$ and $\ib$ are well-centred, then so is $\ia+\ib$.
\end{prop}

\begin{proof}
a) follows from \ref{pre05} D) and \ref{fac63} B). b) Suppose that $\ia$ is centred and $\ib$ is well-centred or weakly fair. Let $M$ be an $R$-module with $\ga_{\ia+\ib}(M)=0$. Then, $\ga_\ia(\ga_\ib(M))=0$ (\ref{pre05} D)). Centredness of $\ia$ yields $\assf_R(\ga_\ib(M))\cap\var(\ia)=\emptyset$. The hypothesis on $\ib$ implies that $\assf_R(M/\ga_\ib(M))\cap\var(\ib)=\emptyset$ (\ref{fac163} c)). It follows that \[\assf_R(M)\cap\var(\ia+\ib)\subseteq(\assf_R(\ga_\ib(M))\cup\assf_R(M/\ga_\ib(M)))\cap\var(\ia)\cap\var(\ib)\subseteq\]\[(\assf_R(\ga_\ib(M))\cap\var(\ia))\cup(\assf_R(M/\ga_\ib(M))\cap\var(\ib))=\emptyset\] (\ref{pre10} B)). Thus, $\ia+\ib$ is centred. c) follows from a) and b).
\end{proof}

\begin{prop}\label{fac170}
Let $M$ be an $R$-module. If $\oga_\ia(M)=0$, then $M$ is large $\ia$-fair. If $\ga_\ia(M)=0$, then $M$ is $\ia$-fair.
\end{prop}

\begin{proof}
This follows immediately from \ref{fac20} a).
\end{proof}

\begin{prop}\label{fac172}
Let\/ $\dim(R)=0$, and let $\ia\subsetneqq R$ be a proper ideal. If $\ga_\ia(R)=0$, then the $R$-module $R$ is not weakly $\ia$-fair.
\end{prop}

\begin{proof}
As $\dim(R)=0$, every prime ideal is minimal over $0=(0:_R1)$, implying $\assf_R(R)=\spec(R)$. It follows \[\assf_R(R)\setminus\var(\ia)\subsetneqq\spec(R)=\assf_R(R)=\assf_R(R/\ga_\ia(R)),\] and thus the claim holds.
\end{proof}

\begin{prop}\label{fac180}
If $\sqrt{\ia}\in\Max(R)$, then: \[\xymatrix{\ia\text{ weakly fair}\ar@{=>}[r]&\ia\text{ half-centred}\ar@{=>}[r]&\ia\text{ weakly quasifair}\ar@{<=>}[r]&\ia\text{ centred.}}\]
\end{prop}

\begin{proof}
The hypothesis implies that $\var(\ia)=\{\sqrt{\ia}\}$. First, let $\ia$ be weakly fair. Let $M$ be an $R$-module with $\assf_R(M)\subseteq\{\sqrt{\ia}\}$. If this inclusion is proper, $M=0$ (\ref{pre10} A)), and otherwise, \[\assf_R(M/\ga_\ia(M))=\assf_R(M)\setminus\{\sqrt{\ia}\}=\emptyset.\] It follows that $\ga_\ia(M)=M$, hence $\ia$ is half-centred. Next, let $\ia$ be half-centred. If $M$ is an $R$-module with $\assf_R(M)\cap\var(\ia)\neq\emptyset$, then $\assf_R(M)=\{\sqrt{\ia}\}$, hence half-centredness implies $\ga_\ia(M)=M$, and so $M$ is weakly $\ia$-quasifair. Therefore, $\ia$ is weakly quasifair. Finally, let $\ia$ be centred. Let $M$ be an $R$-module. We have \[\emptyset\subseteq\assf_R(\ga_\ia(M))\subseteq\assf_R(M)\cap\var(\ia)\subseteq\{\sqrt{\ia}\}\] (\ref{fac10} A) b)). If the first inclusion is proper, then the second one is an equality. If the first inclusion is an equality, then so is the second one by centredness and \ref{pre10} A). Thus, $M$ is weakly $\ia$-quasifair. As weakly quasifair ideals are centred (\ref{fac130} a)), the claim is proven.
\end{proof}

\begin{cor}\label{fac185}
If $\sqrt{\ia}\in\Max(R)$ and $\ga_\ia$ is a radical, then: \[\xymatrix{\ia\text{ half-centred}\ar@{<=>}[r]&\ia\text{ weakly quasifair}\ar@{<=>}[r]&\ia\text{ centred.}}\]
\end{cor}

\begin{proof}
This follows immediately from \ref{fac180} and \ref{fac63} B).
\end{proof}

\begin{cor}\label{fac190}
If the maximal ideal of a $0$-dimensional local ring is weakly quasifair or centred, then every ideal is weakly quasifair and centred.
\end{cor}

\begin{proof}
This follows immediately from \ref{fac180} and \ref{fac132} b).
\end{proof}

\begin{prop}\label{fac200}
If every prime ideal in $\var(\ia)$ is centred, then $\ia$ is centred.
\end{prop}

\begin{proof}
Let $M$ be an $R$-module with $\assf_R(M)\cap\var(\ia)\neq\emptyset$. There exists a centred $\ip\in\assf_R(M)\cap\var(\ia)$. It follows $0\neq\ga_\ip(M)\subseteq\ga_\ia(M)$ (\ref{pre05} C)), hence $\ga_\ia(M)\neq 0$. Therefore, $\ia$ is centred.
\end{proof}


\section{Idempotent ideals}

In this section, we consider fairness and centredness properties of idempotent ideals, and -- as an application -- of ideals in absolutely flat rings. Our main results are that idempotent ideals are fair, and that all ideals in an absolutely flat ring share all the fairness and centredness properties.

\begin{prop}\label{idem10}
a) Let $M$ be an $R$-module of bounded small $\ia$-torsion. If $\ass_R(M/\ga_\ia(M))\cap\var(\ia)=\emptyset$, then $M$ is $\ia$-fair. If $\assf_R(M/\ga_\ia(M))\cap\var(\ia)=\emptyset$, then $M$ is $\ia$-fair and weakly $\ia$-fair.

b) Let $M$ be an $R$-module of bounded large $\ia$-torsion. Then, $M$ is large $\ia$-fair. If $\assf_R(M/\oga_\ia(M))\cap\var(\ia)=\emptyset$, then $M$ is weakly large $\ia$-fair.
\end{prop}

\begin{proof}
a) We prove both claims simultaneously. As there exists $n\in\N$ with $\ga_\ia(M)=(0:_M\ia^n)$, we have $(0:_R\ga_\ia(M))\subseteq\var(\ia)$. Let $\ip\in\ass_R(M/\ga_\ia(M))$ or $\ip\in\assf_R(M/\ga_\ia(M))$. By our hypothesis, $\ip\notin\var(\ia)$, hence $\ip\notin\var(0:_R\ga_\ia(M))$, and so \ref{pre12} implies $\ip\in\ass_R(M)\setminus\var(\ia)$ or $\ip\in\assf_R(M)\setminus\var(\ia)$. Thus, $M$ is $\ia$-fair or weakly $\ia$-fair. This proves the first claim, and together with \ref{pre10} A) we get the remaining part of the second claim. b) follows from a), \ref{pre06} and \ref{fac163} a).
\end{proof}

\begin{cor}\label{idem20}
If $\ga_\ia$ is a radical, every $R$-module of bounded small $\ia$-torsion is $\ia$-fair. If $\ia$ is well-centred, every $R$-module of bounded small $\ia$-torsion is $\ia$-fair and weakly $\ia$-fair.
\end{cor}

\begin{proof}
This follows immediately from \ref{idem10} a) and \ref{fac163} b), c).
\end{proof}

\begin{cor}\label{idem40}
Idempotent ideals are fair.\footnote{This generalises \cite[5.1]{asstor}.}
\end{cor}

\begin{proof}
If $\ia$ is idempotent, every $R$-module is of bounded small $\ia$-torsion and $\ga_\ia$ is a radical (\ref{pre05} E), \ref{pre06}), so \ref{idem20} yields the claim.
\end{proof}

\begin{prop}\label{idem44}
In an absolutely flat ring, every ideal is well-centred, fair and weakly fair.\footnote{This generalises \cite[5.2]{asstor}.}
\end{prop}

\begin{proof}
Let $\ia$ be an ideal in an absolutely flat ring $R$. Then, $\ia$ is half-centred (\cite[4.6 b)]{sol}) and fair (\ref{idem40}). Moreover, every prime ideal is centred (\ref{fac180}), and thus $\ia$ is centred (\ref{fac200}). If $M$ is an $R$-module, then $M$ is of bounded small $\ia$-torsion (\ref{pre06}) and $\assf_R(M/\ga_\ia(M))\cap\var(\ia)=\emptyset$ (\ref{fac163} c)), hence $M$ is weakly $\ia$-fair (\ref{idem10} a)). This shows that $\ia$ is weakly fair, and thus the claim is proven.
\end{proof}

\begin{exas}\label{idem50}
A) Let $K$ be a field, let \[R\dfgl K[(X_i)_{i\in\N}]/\langle X_i^2-X_i\mid i\in\N\rangle,\] denote by $Y_i$ the canonical image of $X_i$ in $R$ for $i\in\N$, and let $\ia\dfgl\langle Y_i\mid i\in\N\rangle_R$. Then, $R$ is absolutely flat, and $\ia$ is maximal and generated by idempotents, but not by a single idempotent (\cite[1.7 C), D)]{sol}). It follows that $\ia$ is well-centred, fair and weakly fair (\ref{idem44}).\smallskip

B) Let $K$ be a field, let \[R\dfgl K[(X_i)_{i\in\Z}]/\langle X_i^2-X_{i+1}\mid i\in\Z\rangle,\] denote by $Y_i$ the canonical image of $X_i$ in $R$ for $i\in\Z$, and let $\ia\dfgl\langle Y_i\mid i\in\Z\rangle_R$. Then, $R$ is a $1$-dimensional Bezout domain, and $\ia$ is maximal and idempotent, but neither generated by idempotents nor half-centred (\cite[1.7 C), E), 5.5 B)]{sol}). It follows that $\ia$ is not centred, and therefore neither weakly quasifair nor weakly fair (\ref{fac185}, \ref{fac130} a)).
\smallskip

C) (cf. \cite[8.3 B)]{sol}) Let $K$ be a field, let \[R\dfgl K[(X_i)_{i\in\N}]/\langle X_i^2-X_i\mid i\in\N^*\rangle,\] let $Y_i$ denote the canonical image of $X_i$ in $R$ for $i\in\N$, and let $\ia\dfgl\langle Y_i\mid i\in\N^*\rangle_R$. Then, $\ia$ is generated by idempotents, hence idempotent, half-centred and fair (\cite[4.6 a)]{sol}, \ref{idem40}). Moreover, $\ia\in\Min(R)$. Indeed, $\ia$ is prime since $R/\ia\cong K[Y_0]\cong K[X_0]$. Let $\ip\in\spec(R)$ with $\ip\subseteq\ia$. Let $i\in\N^*$. If $Y_i\notin\ip$ and $1-Y_i\notin\ip$, then we get the contradiction $0=Y_i(1-Y_i)\notin\ip$. If $1-Y_i\in\ip$, then $1-Y_i\in\ia$, and we get the contradiction $1=1-Y_i+Y_i\in\ia$. It follows that $Y_i\in\ip$, and therefore $\ip=\ia$.

Let $\ib\dfgl\langle Y_0^iY_i\mid i\in\N^*\rangle_R$. Let $f\in R\setminus\{0\}$ with $f\ia\subseteq\ib$, so that there occurs a monomial $g$ in $f$. If $Y_0^p$ is the highest power of $Y_0$ that divides $g$, then we have $fY_{p+1}\in\ib$, hence $gX_{p+1}\in\ib$, and thus $X_0^{p+1}X_{p+1}$ divides $gX_{p+1}$, contradicting our choice of $p$. This shows that $(\ib:_R\ia)=0$, and hence -- as $\ia$ is idempotent -- $\ga_\ia(R/\ib)=0$.

We show now that $\ia$ is not weakly $\ia$-quasifair, and therefore neither centred nor weakly fair (\ref{fac130} b)). Let $g\in(\ib:_RY_0)$. If there occurs in $g$ a monomial of the form $Y_0^l$ with $l\in\N$, then there occurs in $gY_0\in\ib$ a monomial of the form $Y_0^{l+1}$ with $l\in\N$, which is a contradiction. Thus, every monomial occuring in $g$ is a multiple of $Y_i$ for some $i\in\N^*$, and therefore $g\in\ia$. This shows that $(\ib:_RY_0)\subseteq\ia$. As $\ia\in\Min(R)$, we get $\ia\in\min(\ib:_RY_0)=\min(0_{R/\ib}:_R(Y_0+\ib))$, hence $\ia\in\assf_R(R/\ib)$. As $\ga_\ia(R/\ib)=0$, this implies that $R/\ib$ is not weakly $\ia$-quasifair (\ref{fac100}), and therefore our claim holds.

Note that $\ass_R(R/\ib)\cap\var(\ia)=\emptyset$ (\ref{fac60} a)) and hence $\ia\in\assf_R(R/\ib)\setminus\ass_R(R/\ib)$.
\end{exas}

\begin{no}\label{idem60}
A) If $\ia$ is generated by a single idempotent, then it is well-centred, fair and weakly fair (\ref{fac134} a), \ref{idem40}, \cite[1.7 B)]{sol}, \cite[5.5]{asstor}).\smallskip

B) If $\ia$ is generated by idempotents, then it is half-centred and fair (\cite[4.6 a)]{sol}, A)), but it need not be weakly quasifair, hence neither weakly fair nor centred (\ref{fac130} b), \ref{idem50} C)).\smallskip

C) If $\ia$ is idempotent, then it is fair  by A), but it need be neither weakly large quasifair by B) nor half-centred by \ref{idem50} B). Thus, it need be neither weakly quasifair, nor weakly fair, nor weakly large fair, nor centred (\ref{fac130} a)).\smallskip

D) The observations in A)--C) give rise to the following questions:
\begin{qulist}
\item[$(*)$] Are idempotent ideals large fair?
\item[$(**)$] Do there exist a ring $R$ and an ideal $\ia\subseteq R$ that is generated by idempotents, fair, weakly quasifair, but not weakly fair?
\end{qulist}
\end{no}


\section{Nil ideals}

The next class of ideals we turn to are nil ideals. Clearly, they share all large fairness properties, but for small fairness and centredness properties, the situation is more intricate. We will see that for the maximal ideal of a $0$-dimensional local ring there are at least four and at most five possibilities concerning fairness and centredness.

\begin{prop}\label{nil10}
Let $\ia$ be nil.

a) The following statements are equivalent: (i) $\ia$ is nilpotent; (ii) $\ia$ is weakly fair; (iii) $\ia$ is weakly quasifair and $\ga_\ia$ is a radical; (iv) $\ia$ is well-centred; (v) $\ia$ is half-centred.

b) The ideal $\ia$ is fair if and only if $\ass_R(M/\ga_\ia(M))=\emptyset$ for every (monogeneous) $R$-module $M$.

c) The ideal $\ia$ is weakly quasifair if and only if $\assf_R(\ga_\ia(M))=\assf_R(M)$ for every (monogeneous) $R$-module $M$.

d) The ideal $\ia$ is centred if and only if $\ga_\ia(M)\neq 0$ for every nonzero (monogeneous) $R$-module $M$.
\end{prop}

\begin{proof}
First we note that $\var(\ia)=\spec(R)$. a) ``(i)$\Leftrightarrow$(ii)'': The ideal $\ia$ is weakly fair if and only if $\assf_R(M/\ga_\ia(M))=\emptyset$ for every $R$-module $M$, hence if and only if $\ga_\ia=\Id_{\catmod(R)}$ (\ref{pre10} A)), thus if and only if $\ia$ is nilpotent (\ref{pre05} B)). ``(i)$\Rightarrow$(iii)'' follows from \ref{fac134} a) and \ref{fac63} C). ``(iii)$\Rightarrow$(iv)'' follows from \ref{fac130} a) and \ref{fac63} C). ``(iv)$\Rightarrow$(v)'' is clear. ``(v)$\Rightarrow$(i)'': If $\ia$ is half-centred, then $\ga_\ia(M)=M$ for every $R$-module $M$, hence $\ia$ is nilpotent (\ref{pre05} B)). b), c), d) follow immediately from \ref{fac10} C) and \ref{fac70}.
\end{proof}

\newpage
\begin{prop}\label{nil20}
Let $\ia$ be nil.

a) If $\ia$ is idempotent and nonzero, it is neither half-centred nor centred.

b) If $\ga_\ia$ is a radical, $\ia$ is fair.\footnote{This answers a special case of the still open part of \cite[5.11 $(*)$]{asstor}.}
\end{prop}

\begin{proof}
a) As $\ia$ is idempotent, $\ga_\ia$ is a radical (\ref{pre05} E)). Moreover, $\ia$ is not nilpotent, hence not half-centred (\ref{nil10} a)) and thus not centred (\ref{fac63} C)). b) Let $M$ be an $R$-module. We have $\var(\ia)=\spec(R)$ and $\ga_\ia(M/\ga_\ia(M))=0$, hence \[\ass_R(M/\ga_\ia(M))=\ass_R(M/\ga_\ia(M))\cap\var(\ia)=\ass_R(\ga_\ia(M/\ga_\ia(M)))=\emptyset\] (\ref{fac10} A) a)), so \ref{nil10} b) yields the claim.
\end{proof}

\begin{prop}\label{nil25}
Let $R$ be a $0$-dimensional local ring with maximal ideal $\ia$.

a) If $\ga_\ia(R)=0$, then $\ia$ is not centred. If, in addition, there exists an ideal $\ib\subseteq R$ with $\ga_\ia(R/\ib)=\ia/\ib$, then $\ia$ is not fair.

b) If $\ga_\ia(R)=\ia$, then $\ia$ is centred, but not fair.
\end{prop}

\begin{proof}
a) The ideal $\ia$ is not centred by \ref{nil10} d). Concerning the second claim, we have $\ass_R((R/\ib)/\ga_\ia(R/\ib))=\ass_R(R/\ia)=\{\ia\}$ (\ref{pre10} C)), hence $\ia$ is not fair by \ref{nil10} b). b) Let $\ib\subseteq R$ be an ideal with $\ib\subsetneqq\ia$. Then, $0\neq\ia/\ib=\ga_\ia(R)/\ib\subseteq\ga_\ia(R/\ib)$, hence $\ga_\ia(R/\ib)\neq 0$. As $\ga_\ia(R/\ia)=R/\ia\neq 0$ (\ref{pre05} F)) we get that $\ga_\ia(M)\neq 0$ for every nonzero monogeneous $R$-module $M$. Thus, $\ia$ is centred (\ref{nil10} d)). Finally, $\ia$ is not fair by \ref{fac162} a).
\end{proof}

\begin{exas}\label{nil40}
A) Let $K$ be a field, let \[R\dfgl K[(X_i)_{i\in\N}]/\langle X_i^2\mid i\in\N\rangle,\] let $Y_i$ denote the canonical image of $X_i$ in $R$ for $i\in\N$, and let $\ia\dfgl\langle Y_i\mid i\in\N\rangle_R$. Then, $R$ is a $0$-dimensional local ring whose maximal ideal $\ia$ is nil but not nilpotent, and $\ga_\ia$ is not a radical. We clearly have $\ga_\ia(R)=0$, and setting $\ib\dfgl\sum_{i\in\N}Y_i\ia^i\subseteq R$ we have $\ga_\ia(R/\ib)=\ia/\ib$. Thus, $\ia$ is neither fair nor centred (\cite[1.4 A), 5.4 A)]{sol}, \ref{nil25} a)).\smallskip

B) Let $K$ be a field, let \[R\dfgl K[(X_i)_{i\in\N}]/\langle\{X_iX_j\mid i,j\in\N,i\neq j\}\cup\{X_i^{i+1}\mid i\in\N\}\rangle,\] let $Y_i$ denote the canonical image of $X_i$ in $R$ for $i\in\N$, and let $\ia\dfgl\langle Y_i\mid i\in\N\rangle_R$. Then, $R$ is a $0$-dimensional local ring whose maximal ideal $\ia$ is nil but not nilpotent, and $\ga_\ia$ is not a radical. As $\ga_\ia(R)=\ia$ it follows that $\ia$ is centred, but not fair (\cite[1.4 B), 5.4 B)]{sol}, \ref{nil25} b)).\smallskip

C) Let $K$ be a field, let \[R\dfgl K[(X_i)_{i\in\N}]/\langle\{X_i^2\mid i\in\N\}\cup\{X_iX_j\mid i,j\in\N,2i<j\}\rangle,\] let $Y_i$ denote the canonical image of $X_i$ in $R$ for $i\in\N$, and let $\ia\dfgl\langle Y_i\mid i\in\N\rangle_R$. Then, $R$ is a $0$-dimensional local ring whose maximal ideal $\ia$ is nil but not nilpotent, and $\ga_\ia$ is not a radical. As $\ga_\ia(R)=\ia$ it follows that $\ia$ is centred, but not fair (\cite[1.5, 5.4 C)]{sol}, \ref{nil25} b)).\smallskip

D) Let $K$ be a field, let \[R\dfgl K[(X_i)_{i\in\N}]/\langle X_i^{i+1}\mid i\in\N\rangle,\] let $Y_i$ denote the canonical image of $X_i$ in $R$ for $i\in\N$, and let $\ia\dfgl\langle Y_i\mid i\in\N\rangle_R$. Then, $R$ is a $0$-dimensional local ring whose maximal ideal $\ia$ is nil but not nilpotent, and $\ga_\ia$ is not a radical. Moreover, we have $\ga_\ia(R)=0$. Indeed, if $f\in R\setminus\{0\}$ and $n\in\N$ with $\ia^nf=0$, then there occurs a monomial $g$ in $f$, and we have $Y_k^ng=0$ for all $k\in\N$ with $k\geq n$, implying the contradiction that $Y_k^{k+1-n}$ divides $g$ for every $k\in\N$ with $k\geq n$. Setting \[\ib\dfgl\langle Y_iY_j\mid i,j\in\N,i\neq j\rangle_R,\] we have $\ga_\ia(R/\ib)=\ia/\ib$. Thus, $\ia$ is neither fair nor centred (\cite[5.4 D)]{sol}, \ref{nil25} a)).\smallskip

E) Let $K$ be a field, let $Q$ denote the additive monoid of positive rational numbers, let $R\dfgl K[Q]$ denote the algebra of $Q$ over $K$, and let $\{e_\alpha\mid\alpha\in Q\}$ denote its canonical basis. Then, $\im\dfgl\langle e_\alpha\mid\alpha>0\rangle_R$ is a maximal ideal. We consider $S\dfgl R_\im$ and $\inn\dfgl\im_\im$. Then, $S$ is a $1$-dimensional valuation ring with idempotent maximal ideal $\inn$. Let $\ia\dfgl\langle\frac{e_1}{1}\rangle_S$, let $T\dfgl S/\ia$, and let $\ip\dfgl\inn/\ia$. Then, $T$ is a $0$-dimensional local ring whose maximal ideal $\ip$ is idempotent and nonzero (\cite[2.2]{qr}). In particular, $\ga_\ia$ is a radical (\ref{pre05} E)), and thus $\ia$ is fair, but not centred (\ref{idem40}, \ref{nil10} a)).\footnote{This extends \cite[3.9]{asstor}.}
\end{exas}

\begin{no}\label{nil60}
A) Centred ideals whose small torsion functors are radicals are half-centred (\ref{fac63} B)), and half-centred maximal ideals are centred (\ref{fac180}). In general, half-centredness and centredness are independent. Indeed, there exist well-centred ideals (\ref{fac134} a)), half-centred ideals that are not centred (\ref{idem50} C)), centred ideals that are not half-centred (\ref{nil40} B)), and ideals that are neither half-centred nor centred (\ref{nil40} A)).\smallskip

B) The converses of \ref{fac132} c) need not hold. More precisely, there exist a ring $R$ and a well-centred ideal $\ia\subseteq R$ such that $\sqrt{\ia}$ is neither centred nor half-centred. Indeed, since the zero ideal in any ring is well-centred, it suffices to exhibit a $0$-dimensional local ring whose maximal ideal $\sqrt{0}$ is neither centred nor half-centred, which we did in \ref{nil40} A).\smallskip

C) Since nil ideals share all the large fairness properties, the examples in \ref{nil40} together with \ref{fac130} and \ref{nil10} a) show that if $\ia$ is large fair, weakly large fair and weakly large quasifair, then it need not have any of the small fairness or centredness properties.\smallskip

D) We saw in \ref{fac10} C), \ref{fac40} and \ref{fac70} that fairness can be checked on monogeneous $R$-modules. By \ref{nil40} A), there exist a ring $R$ and an ideal $\ia\subseteq R$ such that $\ia$ is not fair, while the $R$-module $R$ is $\ia$-fair (\ref{fac170}). Thus, fairness cannot be checked on the $R$-module $R$ alone.
\end{no}

\begin{no}\label{nil105}
Let $R$ be a $0$-dimensional local ring with maximal ideal $\ia$. By \ref{nil10} a) and \ref{fac180}, the fairness and centredness properties of $\ia$ are determined by whether it is nilpotent, fair, or centred. Thus, $\ia$ lies in precisely one of the five classes specified in the following table.\smallskip

\begin{center}{\small\begin{tabular}{|c|c|c|c|c|c|c|c|}\hline
&$\ia$ nilpot.&$\ga_\ia$ radical&$\ia$ fair&$\ia$ w.fair&$\ia$ w.q.fair&$\ia$ centred&$\ia$ half-centred\\\hline
I&\checkmark&\checkmark&\checkmark&\checkmark&\checkmark&\checkmark&\checkmark\\\hline
II&--&--&\checkmark&--&\checkmark&\checkmark&--\\\hline
III&--&--&--&--&\checkmark&\checkmark&--\\\hline
IV&--&?&\checkmark&--&--&--&--\\\hline
V&--&--&--&--&--&--&--\\\hline
\end{tabular}}\end{center}\smallskip

\noindent(Note that all large fairness properties are always fulfilled.) It follows from \ref{nil10} a) and \ref{nil40} that the classes I, III, IV and V are nonempty. However, we do not know of an example in class II and thus are left with the following question.
\begin{qulist}
\item[$(*)$] Does there exist a $0$-dimensional local ring, whose maximal ideal $\ia$ is fair and centred, but not nilpotent? 
\end{qulist}
Note that a positive answer to $(*)$ implies a negative answer to \cite[3.10 $(**)$]{asstor}. If the converse of \ref{nil20} b) holds, then class II would indeed be empty. Thus:
\begin{qulist}
\item[$(**)$] Suppose that $\ia$ is a nil ideal (or even the maximal ideal of a $0$-dimensional local ring). Does fairness of $\ia$ imply that $\ga_\ia$ is a radical? 
\end{qulist}
\end{no}


\section{Localisation and delocalisation}

While assassins and torsion functors do not behave nicely with respect to localisation, weak assassins do so (\ref{pre11} A), B)). In this final section we exploit this behaviour to prove elementary results on localisation and delocalisation of radicality of torsion functors and of centredness and weak fairness properties of ideals. Finally, we give a criterion for weak large quasifairness using localisation properties of large torsion functors.

\begin{prop}\label{loc05}
We consider the following statements:
(1) $\ga_\ia$ is a radical;
(2) $\ga_{S^{-1}\ia}$ is a radical for every subset $S\subseteq R$;
(3) $\ga_{\ia_\im}$ is a radical for every $\im\in\max(\ia)$.

We have (1)$\Rightarrow$(2)$\Rightarrow$(3). If $\rho^\im_\ia$ (cf.\ \ref{pre11} A)) is an isomorphism for every $\im\in\max(\ia)$, we have (1)$\Leftrightarrow$(2)$\Leftrightarrow$(3).
\end{prop}

\begin{proof}
``(1)$\Rightarrow$(2)'': Suppose that (1) holds. Let $S\subseteq R$ be a subset, and let $M$ be an $S^{-1}R$-module. Then, \[\ga_{S^{-1}\ia}(M/\ga_{S^{-1}\ia}(M))\res_R=\ga_\ia(M\res_R/\ga_\ia(M\res_R))=0\] (\ref{pre05} G)). This implies (2). ``(3)$\Rightarrow$(1)'': Suppose that $\rho^\im_\ia$ is an isomorphism for every $\im\in\max(\ia)$ and that (3) holds. Let $M$ be an $R$-module. If $\im\in\max(\ia)$, then \[\ga_\ia(M/\ga_\ia(M))_\im=\ga_{\ia_\im}(M_\im/\ga_{\ia_\im}(M_\im))=0\] (\ref{pre11} A)). Together with \ref{pre05} F) this implies (1).
\end{proof}

\begin{prop}\label{loc10}
Let $\P$ denote one of the properties of being half-centred, centred, weakly quasifair, or weakly large quasifair. We consider the following statements:
(1) $\ia$ has $\P$;
(2) $S^{-1}\ia$ has $\P$ for every subset $S\subseteq R$;
(3) $\ia_\im$ has $\P$ for every $\im\in\max(\ia)$.

We have (1)$\Rightarrow$(2)$\Rightarrow$(3). If $\rho^\im_\ia$ is an isomorphism for every $\im\in\max(\ia)$, we have (1)$\Leftrightarrow$(2)$\Leftrightarrow$(3).
\end{prop}

\begin{proof}
First, we prove the claim about half-centredness. ``(1)$\Rightarrow$(2)'': Suppose that $\ia$ is half-centred. Let $S\subseteq R$ be a subset, and let $M$ be an $S^{-1}R$-module. Then, \[\ga_{S^{-1}\ia}(M)\res_R=\ga_\ia(M\res_R)=\oga_\ia(M\res_R)=\oga_{S^{-1}\ia}(M)\res_R\] (\ref{pre05} G), \ref{fac63} B)). This implies (2). ``(3)$\Rightarrow$(1)'': Suppose that $\ia_\im$ is half-centred and $\rho^\im_\ia$ is an isomorphism for every $\im\in\max(\ia)$. Let $M$ be an $R$-module. If $\im\in\max(\ia)$, then \[\ga_\ia(M)_\im=\ga_{\ia_\im}(M_\im)=\oga_{\ia_\im}(M_\im)=\oga_\ia(M)_\im\] (\ref{fac63} B)). Together with \ref{pre05} F) this implies $\ga_\ia(M)=\oga_\ia(M)$ and therefore (1).

Second, we prove the claim about centredness. ``(1)$\Rightarrow$(2)'': Suppose that $\ia$ is centred. Let $S\subseteq R$ be a subset, and let $M$ be an $S^{-1}R$-module with $\ga_{S^{-1}\ia}(M)=0$. Then, $\ga_\ia(M\res_R)=\ga_{S^{-1}\ia}(M)\res_R=0$ (\ref{pre05} G)), hence, by centredness,\pagebreak \[\assf_{S^{-1}R}(M)\cap\var(S^{-1}\ia)\cong\{\ip\in\assf_R(M\res_R)\mid\ip\cap S=\emptyset\}\cap\var(\ia)\subseteq\]\[\assf_R(M\res_R)\cap\var(\ia)=\emptyset\] (\ref{pre11} B)). This implies (2). ``(3)$\Rightarrow$(1)'': Suppose that $\ia_\im$ is centred and $\rho^\im_\ia$ is an isomorphism for every $\im\in\max(\ia)$. Let $M$ be an $R$-module with $\ga_\ia(M)=0$.\linebreak If $\im\in\max(\ia)$, then $\ga_{\ia_\im}(M_\im)\cong\ga_\ia(M)_\im=0$, hence, by centredness,\linebreak $\assf_{R_\im}(M_\im)\cap\var(\ia_\im)=\emptyset$. Together with \ref{pre11} B) this implies $\assf_R(M)\cap\var(\ia)=\emptyset$, and therefore (1) holds.

Third, we prove simultaneously the claims about weak quasifairness and weak large quasifairness. For an ideal $\ib$ we write $F_\ib$ for $\ga_\ib$ or $\oga_\ib$. ``(1)$\Rightarrow$(2)'': Suppose that $\ia$ is weakly (large) quasifair. Let $S\subseteq R$ be a subset. Let $M$ be an $S^{-1}R$-module. If $\mathfrak{P}\in\assf_{S^{-1}R}(S^{-1}M)\cap\var(S^{-1}\ia)$, then there exists $\ip\in\linebreak\assf_R(M)\cap\var(\ia)=\assf_R(F_\ia(M))$ with $\ip\cap S=\emptyset$ and \[\mathfrak{P}=S^{-1}\ip\in\assf_{S^{-1}R}(S^{-1}F_\ia(M))\subseteq\assf_{S^{-1}R}(F_{S^{-1}\ia}(S^{-1}M))\] (\ref{pre10} B), \ref{pre11} A), B)), implying (2). 

``(3)$\Rightarrow$(1)'': Suppose that $\ia_\im$ is weakly (large) quasifair and that $\rho^\im_\ia$ (or $\overline{\rho}^\im_\ia$) is an isomorphism for every $\im\in\max(\ia)$. Let $M$ be an $R$-module. Let $\ip\in\assf_R(M)\cap\var(\ia)$. If $\im\in\max(\ia)$ with $\ip\subseteq\im$, then \[\ip_\im\in\assf_{R_\im}(M_\im)\cap\var(\ia_\im)=\assf_{R_\im}(F_{\ia_\im}(M_\im))=\assf_{R_\im}(F_\ia(M)_\im)\] (\ref{pre11} B)), hence $\ip\in\assf_R(F_\ia(M))$, implying (1).
\end{proof}

\begin{prop}\label{loc40}
a) Let $S\subseteq R$ be a subset such that $\rho^S_\ia$ (or $\overline{\rho}^S_\ia$) is an isomorphism. If $\ia$ is weakly (large) fair, then so is $S^{-1}\ia$.

b) Suppose that $\rho_\ia^\im$ (or $\overline{\rho}^\im_\ia$) is an isomorphism for every $\im\in\Max(R)$. Then, $\ia$ is weakly (large) fair if and only if $\ia_\im$ is weakly (large) fair for every $\im\in\Max(R)$.
\end{prop}

\begin{proof}
For an ideal $\ib$ we write $F_\ib$ for $\ga_\ib$ or $\oga_\ib$. a) Let $M$ be an $S^{-1}R$-module. If \[\mathfrak{P}\in\assf_{S^{-1}R}(S^{-1}M/F_{S^{-1}\ia}(S^{-1}M))=\assf_{S^{-1}R}(S^{-1}(M/F_\ia(M))),\] then there exists $\ip\in\assf_R(M/F_\ia(M))=\assf_R(M)\setminus\var(\ia)$ with $\ip\cap S=\emptyset$ and $\mathfrak{P}=S^{-1}\ip\in\assf_{S^{-1}R}(S^{-1}M)\setminus\var(S^{-1}\ia)$ (\ref{pre11} B)), implying the claim. b) Suppose that $\ia_\im$ is weakly (large) fair for every $\im\in\Max(R)$. Let $M$ be an $R$-module, and let $\ip\in\assf_R(M/F_\ia(M))$. If $\im\in\Max(R)$ with $\ip\subseteq\im$, then \[\ip_\im\in\assf_{R_\im}((M/F_\ia(M))_\im)=\assf_{R_\im}(M_\im/F_{\ia_\im}(M_\im))=\assf_{R_\im}(M_\im)\setminus\var(\ia_\im),\] hence $\ip\in\assf_R(M)\setminus\var(\ia)$ (\ref{pre11} B)). Together with a) this yields the claim.
\end{proof}

\begin{prop}\label{loc50}
If $\overline{\rho}^\ip_\ia$ is an isomorphism for every $\ip\in\var(\ia)$, then $\ia$ is weakly large quasifair.
\end{prop}

\begin{proof}
Let $M$ be an $R$-module with $\oga_\ia(M)=0$. If $\ip\in\var(\ia)$, then $\oga_{\ia_\ip}(M_\ip)=\oga_\ia(M)_\ip=0$, hence $\assf_{R_\ip}(M_\ip)\cap\var(\ia_\ip)=\emptyset$ (\ref{fac60} a)), thus $\ip_\ip\notin\assf_{R_\ip}(M_\ip)$, and therefore $\ip\notin\assf_R(M)$ (\ref{pre11} B)). It follows that $\assf_R(M)\cap\var(\ia)=\emptyset$, and so \ref{fac100} yields the claim.
\end{proof}

\noindent\textbf{Acknowledgement:} I thank the anonymous referee for his suggestions.


\begin{thebibliography}{99}

\bibitem{lipman} L. Alonso Tarr\'io, A. Jerem\'ias L\'opez, J. Lipman, \textit{Local homology and cohomology on schemes.} Ann. Sci. \'Ec. Norm. Sup\'er. (4) 30 (1997), 1--39.

\bibitem{ac} N. Bourbaki, \textit{\'El\'ements de math\'ematique. Alg\`ebre commutative. Chapitres 1 \`a 4.} Masson, Paris, 1985.

\bibitem{bs} M. P. Brodmann, R. Y. Sharp, \textit{Local cohomology (second edition).} Cambridge Stud. Adv. Math. 136. Cambridge Univ. Press, Cambridge, 2013.

\bibitem{qr} P. H. Quy, F. Rohrer, \textit{Injective modules and torsion functors.} Comm. Algebra 45 (2017), 285--298.

\bibitem{asstor} F. Rohrer, \textit{Assassins and torsion functors.} Acta Math. Vietnam. 43 (2018), 125--136.

\bibitem{sol} F. Rohrer, \textit{Torsion functors, small or large.} Beitr. Algebra Geom. 60 (2019), 233--256.

\bibitem{schenzel} P. Schenzel, A.-M. Simon, \textit{Examples of injective modules for specific rings.} In preparation.

\bibitem{stacks} The Stacks Project Authors, \textit{Stacks project.} {\tt https://stacks.math.columbia.edu}

\bibitem{yassemi} S. Yassemi, \textit{Coassociated primes of modules over a commutative ring.} Math. Scand. 80 (1997), 175--187.

\end{thebibliography}
\end{document}